\documentclass[11pt,letterpaper]{amsart}

\usepackage{latexsym}

\usepackage{amssymb} 
\usepackage{amsthm} 
\usepackage{amsfonts} 
\usepackage{amsmath} 
\usepackage{amstext} 
\usepackage{amscd} 
\usepackage[all]{xy}
\usepackage{color}

\numberwithin{equation}{section} 

\theoremstyle{plain}
\newtheorem{thm}{Theorem}[section] 
\newtheorem{prop}[thm]{Proposition} 
\newtheorem{cor}[thm]{Corollary} 
\newtheorem{lem}[thm]{Lemma}

\newtheorem{theorem*}{Theorem}[] 

\theoremstyle{definition}

\theoremstyle{remark} 
\newtheorem{rem}[thm]{Remark} 
\newtheorem{question}[thm]{Question}

\newcommand{\R}{\mathbb{R}} 

\newcommand{\Z}{\mathbb{Z}}

\newcommand{\proj}{\mathbb{P}}

\newcommand{\AS}{\mathcal{AS}}
\newcommand{\as}[1]{\overline{#1}^{\AS}}

\newcommand{\inv}{^{-1}} 

\newcommand{\p}{\partial}
\newcommand{\mc}{\mathcal}

\DeclareMathOperator{\lk}{lk} 
\DeclareMathOperator{\half}{\textstyle\frac 1 2 } 
\DeclareMathOperator{\codim}{codim} 
\DeclareMathOperator{\Ker}{Ker}
\DeclareMathOperator{\Coker}{Coker}
\DeclareMathOperator{\Image}{Im}

\DeclareMathOperator{\ih}{IH}

\DeclareMathOperator{\cl}{cl}
\DeclareMathOperator{\supp}{Supp}
\DeclareMathOperator{\bd}{bd}

\begin{document} 

\address {Mathematics Department, University of Georgia, Athens GA
30602, USA}

\email{clint@math.uga.edu}

\address {Universit\' e Nice Sophia Antipolis, CNRS,  LJAD, UMR 7351, 06108 Nice, France}

\email{adam.parusinski@unice.fr}


\title{Real intersection homology}

\author{ Clint McCrory and Adam Parusi\'nski }

\begin{abstract}
We present a definition of intersection homology for real algebraic varieties that is analogous to Goresky and MacPherson's original definition of intersection homology for complex varieties.
 \end{abstract}

\maketitle


Let $X$ be a real algebraic variety. For certain stratifications $\mc S$ of $X$ we define homology groups $\ih _k ^{\mc S}(X)$ with $\Z/2$ coefficients that generalize the standard intersection homology groups \cite{gm1} if all strata have even codimension. Whether there is a good analog of intersection homology for real algebraic varieties was stated as a problem by Goresky and MacPherson \cite{gm3} (Problem 7, p.\ 227). They observed that if such a theory exists then it cannot be purely topological; indeed our groups are not homeomorphism invariants.

We consider a class of algebraic stratifications introduced in \cite{mpp} that have a natural general position property for semialgebraic subsets. After presenting the definition and properties of these stratifications $\mc S$, we define the real intersection homology groups $\ih _k ^{\mc S}(X)$ and show that they are independent of the stratification. We prove that if $X$ is nonsingular and pure dimensional then $\ih_k(X) = H_k(X;\Z/2)$, classical homology with $\Z/2$ coefficients. More generally, we prove that if $X$ is irreducible and $X$ admits a small resolution $\pi : \widetilde X\to X$ then $\ih_k(X)$ is canonically isomorphic to $H_k (\widetilde X;\Z/2)$. Thus any two small resolution of $X$ have the same homology. 

If $X$ is not compact, we have two versions of real intersection homology: $\ih^c_k(X)$ with compact supports and $\ih^{cl}_k(X)$ with closed supports.
We define an intersection pairing $
\ih^c_k(X) \times \ih^{cl}_{n-k}(X) \to \Z/2$, where $n=\dim X$. We prove that if $X$ has isolated singularities this pairing is nonsingular,  so $\ih^c_k(X) \cong \ih^{cl}_{n-k}(X)$ for all $k\geq 0$. But our intersection pairing is singular for some real algebraic varieties $X$, so our groups fail to have the key self-duality property suggested by Goresky and MacPherson. We will present a counterexample in a subsequent paper.

We benefitted from conversations more than a decade ago with Joost van Hamel, whose approach to real intersection homology \cite{vanhamel} was quite different from ours. We dedicate this paper to his memory.


\section{Good algebraic stratifications}

An  \emph{algebraic stratification}  of a real algebraic variety $X$ is the stratification associated to a filtration of $X$ by algebraic subvarieties 
$$X=X_n\supset X_{n-1} \supset \cdots \supset X_0 \supset X_{-1} = \emptyset $$
    such that for each $j$, $\operatorname{Sing}(X_j) \subset X_{j-1}$, and either $\dim X_j = j$ or $X_j = X_{j-1}$.  
This filtration induces a decomposition $X=\bigsqcup S_i$, where the $S_i$ are the connected components of all the semialgebraic sets $X_j\setminus  X_{j-1}$.  The sets $S_i$, which are nonsingular semialgebraic subsets of $X$, are  
the strata of the algebraic stratification $\mc S= \{S_i\}$ 
of $X$. (For this definition we do not assume the frontier condition or any regularity conditions for pairs of incident strata.)

Let $\mc T$ be an algebraic stratification of projective space $\proj^r$. An $l$-parameter \emph{submersive deformation of the identity} of $(\proj^r,\mc T)$ is a semialgebraic stratified submersive family of diffeomorphisms $\Psi: U\times \proj^r\to \proj^r$, where $U$ is a semialgebraic open neighborhood of the origin in $\R^l$, such that $\Psi(0,x) = x$ for all $x\in \proj^r$, and the map $\Phi: U\times \proj^r\to U\times \proj^r$, $\Phi(t,x) = (t,\Psi(t,x))$ is an arc-wise analytic trivialization of the projection $U\times \proj^r\to U$. 

In particular, for each stratum $S\in \mc T$, we have $\Psi(U\times S) \subset S$. For the restriction $\Psi: U\times S\to S$ we have that $\Psi_t(x) = \Psi(t,x)$ is a diffeomorphism for all $t\in U$, and if $\Psi^x(t) = \Psi(t,x)$ then the differential $D\Psi^x$ at $t$ is surjective for all $x\in S$. Moreover, for all $t\in U$ the map $\Psi_t: \proj^r\to \proj^r$ is an  arc-analytic stratum-preserving semialgebraic homeomorphism with arc-analytic inverse. (For further details see \cite{mpp}, \cite{pp}.)

We say the algebraic stratification $\mc S$ of $X\subset \proj^r$ is \emph{good} if it is a Whitney stratification and there exists an algebraic Whitney stratification $\mc T$ of $\proj^r$ such that $\mc S = \mc T|X$ (\emph{i.e.} $\mc S = \{ S\in \mc T; S\subset X\}$) and there exists a submersive deformation of the identity of $(\proj^r,\mc T)$.

With this terminology, the main theorem of \cite{mpp} gives the following existence theorem.

\begin{thm}\label{goodexist}
Let $X$ be a subvariety of $\proj^r$, and let $\mc V$ be a finite family of algebraic subvarieties of $X$. There exists a good stratification $\mc S$ of $X$ compatible with all the subvarieties $V\in\mc V$.
\end{thm}

\begin{cor}\label{goodrefine} Let $X$ be a subvariety of $\proj^r$.

(a) Every algebraic stratification of $X$ has a good refinement.

(b) Any two good algebraic stratifications of $X$ have a common good refinement.
\end{cor}
\begin{proof}
(a) If the  stratification is associated to the filtration $X=X_n\supset X_{n-1} \supset \cdots \supset X_0$,
we apply the theorem with $\mc V =\{X_i\}$.

(b) If the  stratifications $\mc S$ and $\mc S'$ are associated to the filtrations $\{X_i\}$ and $\{X'_j\}$, respectively, we apply the theorem with $\mc V =\{X_i\cap X'_j\}$.
\end{proof}

Recall the  \emph{frontier condition} for a stratification $\mc S = \{S_i\}$: If $S_1$ and $S_2$ are disjoint strata of $\mathcal S$ such that $S_1$ intersects the closure of $S_2$, then $S_1$ is contained in the closure of $S_2$.

\begin{prop}\label{goodrefinesemialg} Let $X$ be a subvariety of $\proj^r$.
Every semialgebraic stratification $\mc S$ of $X$ satisfying the frontier condition has a good refinement.
\end{prop}
\begin{proof}
Let $X=X_n\supset X_{n-1} \supset \cdots \supset X_0$ be the filtration of $X$ by the skeletons of $\mc S$, where $\dim X_j = j$ or $X_j = X_{j-1}$. Let $\overline X_j$ be the Zariski closure of $X_j$ and let $\mc V = \{\overline X_j\}$. By Theorem \ref{goodexist} there is a good stratification $\mc R$ of $X$ compatible with $\mc V$. We claim that $\mc R$ is a refinement of $\mc S$. 

We show by induction on $k$ that every stratum of $\mc S$ of dimension $\leq k$ is a union of strata of $\mc R$. This is true for $k=0$, since $\overline X_0 = X_0$. Now suppose that every stratum of $\mc S$ of dimension $< k$ is a union of strata of $\mc R$, and let $S$ be a $k$-dimensional stratum of $\mc S$. We have that $\bd S=\cl S \setminus S$ is a union of strata of $\mc R$, by the frontier condition for $\mc S$ and inductive hypothesis.

Consider the partition $P_S$ of $S$ induced by the strata of $\mc R$. We have $P_S = \bigsqcup U_i$, where $U_i$ is a connected component of $R_i\cap S$, with $R_i$ a (connected) stratum of $\mc R$. 
Let $U$ be a non-empty $k$-dimensional element of $P_S$. Then $U \subset R\cap S$, where $R$ is a stratum of $\mc R$ of dimension $\geq k$. We claim that $R=U$. 
Now $U\subset X_k \subset \overline X_k$, so if $\dim R>k$ then $R$ is not compatible with $\mc V=\{\overline X_j\}$. Thus $\dim R = k$. So $U$ is an open connected nonsingular $k$-dimensional subset of the connected nonsingular $k$-dimensional semialgebraic set $R$. Consider $\bd U = \cl U\setminus U$. Now $\bd U \subset \cl S = S\cup \bd S$. We have $(\bd U\cap R)\cap \bd S = \emptyset$, since $R\cap \bd S = \emptyset$. On the other hand, $(\bd U\cap R)\cap S = \bd U\cap (R\cap S)=\emptyset$, since $U$ is a connected component of  $R\cap S$. Therefore $\bd U \cap R = \emptyset$, which implies that $R=U$, since $U$ is a connected open subset of the connected set $R$.

Now suppose that $U'$ is a non-empty element of $P_S$ of dimension $<k$. Then $U' \subset R'\cap S$, where $R'$ is a stratum of $\mc R$. Now $U'\cap \cl U\neq\emptyset$ for some $k$-dimensional element $U$ of $P_S$. Since we have shown $U\in \mc R$, and $\mc R$ satisfies the frontier condition, we conclude that $R'$ is contained in the closure of $U$. Thus $R'$ is a connected subset of $S$, so $R' = U'$.
\end{proof}

\begin{cor}\label{semigood}
Let $X$ be a subvariety of $\proj^r$, and let $\mc V$ be a finite family of semialgebraic subsets of $X$. There exists a good stratification $\mc S$ of $X$ compatible with all  $V\in\mc V$.\qed
\end{cor}

\begin{prop}\label{genpos}
Let $\mc S$ be a good stratification of $X\subset \proj^r$, and let $\Psi: U\times \proj^r\to \proj^r$ be a submersive deformation of the identity of $(\proj^r,\mc T)$, where $\mc S = \mc T|X$.  Let $\Psi_t(x) = \Psi(t,x)$. Let $Z$ and $W$ be semialgebraic subsets of $X$.

(a)  There is an open dense semialgebraic subset $U'$ of $U$ such that, for all $t \in U'$ and all strata $S\in \mc S$,
 $$
\dim   (\Psi_t(Z) \cap W \cap S) \le  \dim (Z\cap S)  + \dim (W\cap S)  - \dim S. 
$$

(b) Suppose there are semialgebraic stratifications $\mc A$ of $Z$ and $\mc B$ of $W$ such that $(Z,\mc A)$ and $(W,\mc B)$ are substratified objects of $(X,\mc S)$. There is an open dense semialgebraic subset $U'$ of $U$ such that, for all $t \in U'$, $\Psi_t(Z,\mc A)$ is transverse to $(W,\mc B)$ in $(X,\mc S)$.
\end{prop}

\begin{proof}
This follows from Propositions 1.3 and 1.5 of \cite{mpp}.
\end{proof}

\begin{rem}
It follows from \cite{mpp} (Cor. 3.2) that all the results of this section hold for affine real algebraic varieties $X$, replacing $\proj^r$ by $\R^r$ throughout.
\end{rem}


\section{Allowable chains and intersection homology} 
\label{Chains} 

Let $X$ be a semialgebraic set. The complex of semialgebraic chains of $X$ has the following geometric definition (see \cite{mccpar2}, Appendix). For $k\geq 0$, let $S_k(X)$ be the $\Z/2$ vector space generated by the closed semialgebraic subsets of $X$ of dimension $\leq k$. The chain group $C_k(X)$ is the quotient of $S_k(X)$ by the following relations:

\vskip.1in
(i) If $A$ and $B$ are closed semialgebraic subsets of $X$ of dimension at most $k$ then
$$
A+B \sim \cl(A\div B),
$$
where $A\div B = (A\cup B)\setminus(A\cap B) = (A\setminus B)\cup (B\setminus A)$ is the symmetric difference of $A$ and $B$, and $\cl$ denotes closure.

(ii) If $A$ is a closed semialgebraic subset of $X$ and $\dim A < k$, then $A\sim 0$.

\vskip.1in
If the $k$-chain $\Gamma$ is represented by the semialgebraic set $C$, we write $\Gamma =[C]$. The \emph{support} of $\Gamma$, denoted $\supp \Gamma$, is the smallest closed semialgebraic set representing $\Gamma$. If $\Gamma=[C]$ then $\supp \Gamma = \{x\in C\ ;\ \dim_x C = k \}$. If $X$ is not compact, we may consider the group $C^c_k(X)$ of $k$-chains with compact support, which is a subgroup of $C_k(X)$.

\emph{Note}: To simplify notation, we often identify the $k$-dimensional semialgebraic set $C$ with its equivalence class in the group of chains. So we may refer to  $C$ as a ``$k$-chain,'' or as a ``$k$-cycle'' if $\p C\sim 0$. 

The boundary operator $\p_k: C_k(X)\to C_{k-1}(X)$ is defined using the Euler characteristic of the link: $\p[C] = [\p C]$, where
$$
\p C = \supp[\cl\{x\in C\ ;\  \chi (\lk (C,x)) \not
\equiv 0 \pmod 2\}]
$$
\begin{rem}
This corrects the definition of $\p C$ in \cite{mccpar2}, where we omitted the closure operator. Adding the support to the definition is not essential, but it will simplify several arguments.
\end{rem}
The following characterization of the boundary operator follows from the definition.
\begin{prop}
If $\mathcal T$ is a semialgebraic triangulation of the closed $k$-dimensional semialgebraic set $C$, then $\p C$ is the union of the closed $(k-1)$-simplices $S\in\mathcal T$ such that the number of $k$-simplices of $\mathcal S$ incident to $S$ is odd.\qed
\end{prop}

The following result shows that the boundary operator is well-defined on semialgebraic chains. (We overlooked this result in \cite{mccpar2}.)

\begin{prop}
$\p\cl(A\div B) =  \cl(\p A\div \p B)$.
\end{prop}
\begin{proof}
Let $\dim A = \dim B=k$. The sets $\p\cl(A\div B)$ and $\cl(\p A\div \p B)$ are contained in $A\cup B$. Let $\mathcal T$ be a semialgebraic triangulation of $A\cup B$ such that $A$ and $B$ are unions of simplices. Let $S$ be a $(k-1)$-simplex of $\mathcal T$. If $Y$ is a union of simplices of $\mathcal T$, let $\langle Y,S\rangle$ denote the number of $k$-simplices of $Y$ incident to $S$. Let $a=\langle A,S\rangle$, $b=\langle B,S\rangle$, and $c=\langle A\cap B, S\rangle$. Then $\langle A\setminus B, S\rangle = a-c$ and $\langle B\setminus A, S\rangle = b-c$, so $\langle A\div B, S \rangle = (a-c) + (b-c) = a+b-2c$. So $S\subset\p\cl(A\div B)$ if and only if $a+b$ is odd. On the other hand, $S\subset\p A$ if and only if $a$ is odd, and $S\subset\p B$ if and only if $b$ is odd. So $S\subset\cl(\p A\div \p B)$ if and only if $a+b$ is odd.
\end{proof}

We have $\p(C^c_k(X))\subset C^c_{k-1}(X)$, and the homology of the chain complex $(C^c_*(X), \p)$ is canonically isomorphic to the simplicial  homology $H_*(X;\Z/2)$ with respect to a (locally finite) semialgebraic triangulation. The homology of $(C_*(X), \p)$ is canonically isomorphic to the homology $H^{cl}_*(X;\Z/2)$ with closed supports (Borel-Moore homology) (\emph{cf.} \cite{bcr} \S 11.7).

Let $Y$ be a closed semialgebraic subset of $X$. The relative homology of $X$ mod $Y$ can be described as follows. Let $C^c_k(Y)$ be the subcomplex of $C^c_k(X)$ consisting of semialgebraic chains $\Gamma$ of $X$ such that $\supp \Gamma \subset Y$. Let $C^c_k(X,Y) = C^c_k(X)/C^c_k(Y)$. The homology of the chain complex $C^c_k(X,Y)$ is canonically isomorphic to the relative simplicial homology group $H_k(X,Y;\Z/2)$.

\begin{prop}\label{openH}
Let $V$ be a semialgebraic open subset of the compact semialgebraic set $X$. For $k\geq 0$ there is a canonical isomorphism 
$$
H^{cl}_k(V;\Z/2) \cong H_k(X,X\setminus V;\Z/2).
$$
\end{prop}
\begin{proof}
If $C$ is a closed semialgebraic subset of $V$, let $\overline C$ be the closure of $C$ in $X$. The function $C\mapsto \overline C$ induces an isomorphism of chain complexes $C^{cl}_*(V)\cong C^c_*(X,X\setminus V)$.
\end{proof}

For the rest of this section we assume that $X$ is a real algebraic variety.

\vskip.1in

Arc-symmetric sets were introduced by Kurdyka \cite{kurdyka1} (see also \cite{kp}). Let $C\subset X$ be a closed semialgebraic set of dimension $k$. We say that $C$ is {\em arc-symmetric in dimension} $k$ at 
$x\in C$ if there is an open neighborhood $U$ of $x$ in $C$ and 
a semialgebraic set $V\subset U$, $\dim V < k$, such that for
every real analytic curve $\gamma : (-1,1) \to X$, if 
$\gamma ((-1,0)) \subset U\setminus V$ and $\gamma (0)\in U$,  
then $\gamma ((0,\varepsilon)) \subset U\setminus V$ for some 
$\varepsilon >0$.  In terms of the arc-symmetric closure of
\cite{kurdyka1} this
condition means that $\dim _x (\as C \setminus C)<k$.  

For instance the closure of the regular part of an algebraic 
set $Y$ is always arc-symmetric in $\dim Y$ at every point 
of $Y$, but $Y$ is not necessarily arc-symmetric (\emph{e.g.}\ the Whitney umbrella $wx^2=yz^2$).  A segment or a half-line are examples of semialgebraic sets that are not arc-symmetric in dimension $1$.  

If $C$ is a $k$-dimensional closed semialgebraic subset of $X$, we define the \emph{pseudoboundary} by 
\begin{equation*}
\Sigma C = \{x\in C\ ; \text { $C$ is not arc-symmetric in 
dimension $k$ at $x$}\}.
\end{equation*} 
$\Sigma C$ is a closed semialgebraic set,  and $\dim \Sigma C < 
\dim C$.  By Proposition \ref{boundary} of the Appendix,  
\begin{equation}\label{sigma}
\p C \subset \Sigma C.
\end{equation}
 Moreover, $\Sigma C$ depends only on $[C] \in C_k(X)$, since  $\Sigma C$ is determined by the support of $[C]$. Note that $\Sigma (\Sigma C)$ is not necessarily empty (consider for instance a quadrant of $\R^2$).

Let $X$ be a subvariety of $\proj^r$ (or $\R^r$), and let $\mc S$ be a good stratification of $X$.  
For a closed semialgebraic set  $C$ of dimension $k$ we consider the following \emph{perversity  
conditions} with respect to the stratification $\mc S$:
\begin{align}\label{perv1}
\dim C\cap S &< \dim C - \half \codim S \\ 
\label{perv2}
\dim \Sigma C \cap S &\le \dim C - \half \codim S - 1 
\end{align}
for all strata $S$ such that $\codim S = \dim X - \dim S >0$.
(A set is empty if it has negative dimension.)  If $\codim S$  is even then 
\eqref{perv2} is a consequence of \eqref{perv1}, which is 
the standard middle perversity condition for intersection homology \cite{gm1}.

A semialgebraic $k$-chain $\Gamma$ of $X$ 
is called  \emph{allowable} (with respect to the stratification 
$\mc S$) if $\Gamma = [C]$, where $C$ and $\p C$ 
satisfy the above perversity conditions.  
The set of allowable $k$-chains (resp., allowable $k$-chains with compact support) is a subgroup of $C_k(X)$ (resp., $C^c_k(X)$). Since $\cl(A\div B) \subset A\cup B$, if two $k$-chains satisfy (\ref{perv1}), then so does their sum. Furthermore, $\Sigma\cl(A\div B)\subset \Sigma A\cup \Sigma B$, so if two $k$-chains satisfy (\ref{perv2}) their sum does also. 

By definition the boundary operator $\p$ 
preserves allowable chains. The \emph{intersection homology groups} $\ih _k ^{c,\mc S} (X)$ (resp., $\ih _k ^{cl,\mc S} (X)$), $k=0 ,\dots,n$, are the homology groups of the complex of allowable chains with compact supports (resp., closed supports) with respect to the good stratification $\mc S$.
If $X$ is a subvariety of $\proj^r$ these groups are equal, and we write $\ih _k ^{\mc S} (X)$.

\begin{question} (Andrzej Weber) Are these intersection homology groups finitely generated? The present paper shows that these groups are finitely generated in several special cases: when $X$ is nonsingular (Theorem \ref{nashhomology}), for homology with compact supports when $X$ has a small resolution (Theorem \ref{smallres}), and when $X$ has isolated singularities (formulas \eqref{even} and \eqref{odd}).
\end{question}

Let $V$ be a semialgebraic open subset of the variety $X$, and let $\mc S$ be a good stratification of $X$ such that $V$ is a union of strata of $\mc S$ (Corollary \ref{semigood}). The intersection homology groups $\ih _k ^{c,\mc S} (V)$ (resp., $\ih _k ^{cl,\mc S} (V)$), $k=0 ,\dots,n$ are defined as above, with chains represented by compact semialgebraic subsets $C$ of $V$ (resp., closed semialgebraic subsets $C$ of $V$).

Let $Y$ be a closed semialgebraic subset of $X$, and let $\mc S$ be a good stratification of $X$ such that $Y$ is a union of strata of $\mc S$.  We define the relative intersection homology (with respect to $\mc S$) of $X$ mod $Y$  as follows. Let $C^{c,\mc S}_k(X)_{X\setminus Y}$ be the subcomplex of $C^c_k(X)$ consisting of semialgebraic chains that satisfy the perversity conditions \eqref{perv1} and \eqref{perv2} for all strata $S\in\mc S$ such that $S\subset X\setminus Y$. Let $C^{c,\mc S}_k(X,Y) = C^{c,\mc S}_k(X)_{X\setminus Y}/C^c_k(Y)$. We define $\ih _k ^{c,\mc S} (X,Y)$ to be the homology of the chain complex $C^{c,\mc S}_*(X,Y)$.

\begin{prop}\label{openIH}
Let $V$ be a semialgebraic open subset of the compact real algebraic variety $X$, and let $\mc S$ be a good stratification of $X$ such that $V$ is a union of strata of $\mc S$.  For $k\geq 0$ there is a canonical isomorphism 
$$
\ih _k ^{cl,\mc S} (V) \cong \ih _k ^{c,\mc S} (X,X\setminus V).
$$
\end{prop}
\begin{proof}
If $C$ is an allowable  closed semialgebraic subset of $V$, let $\overline C$ be the closure of $C$ in $X$. The function $C\mapsto \overline C$ induces an isomorphism of chain complexes $C^{cl,\mc S}_*(V)\cong C^{c,\mc S}_*(X,X\setminus V)$.
\end{proof}

In general position or transversality arguments, we will use the following construction in several different settings. Let $X$ be a real algebraic variety, and let $U$ be a neighborhood of the origin in $\R^l$.
Let $\Phi(t,x) = (t, \Psi (t,x)) : U\times X \to U\times X$  be an arc-analytic and semialgebraic homeomorphism with arc-analytic and semialgebraic inverse, such that $\Psi(0,x) = x$ for all $x\in X$. Let $t_0\in U$ be such that the segment $I=\{t\in\R^l\ ;\ t = st_0, 0\leq s\leq 1\}$ is contained in $U$. Let $C$ be a $k$-cycle in $X$. The $t_0$-\emph{deformation homology} of $C$ with respect to the deformation $\Psi_t$ of the identity of $X$ is the chain 
\begin{equation}\label{defhom}
B = \Psi_* (I\times C),
\end{equation}
\begin{equation*} 
\partial B = C + \Psi_{t_0}(C).
\end{equation*}

\begin{thm}\label{refinement}
Let $X$ be a subvariety of $\proj^r$ and let $\mc S$ be a good stratification of $X$. The groups $\ih _k ^{\mc S} (X)$, $k=0 ,\dots,n$, do not depend on the stratification $\mc S$ or on the embedding $X\subset \proj^r$.
\end{thm}

\begin{proof}
To show independence of the good stratification $\mc S$, it suffices by Corollary \ref{goodrefine} (b) to prove that if $\mc S'$ is a good refinement of $\mc S$, then $\ih _k ^{\mc S'} (X)\cong \ih _k ^{\mc S} (X)$. Since $\mc S'$ is a refinement of $\mc S$, if a $k$-chain is allowable for $\mc S'$ then it is allowable for $\mc S$. (If $S'\subset S$ then $\codim S' \geq \codim S$.) Thus there is a canonical homomorphism $\varphi_k: \ih _k ^{\mc S'} (X)\to \ih _k ^{\mc S} (X)$.

We use general position to prove that $\varphi_k$ is an isomorphism. Let $C$ be a $k$-cycle that is allowable for the stratification $\mc S$. By Proposition \ref{genpos} (a) there is an  arc-analytic $\mc S$-stratum-preserving semialgebraic homeomorphism $\Psi_t: \proj^r\to \proj^r$ with arc-analytic inverse such that for all pairs of strata $S\in \mc S$ and $S'\in \mc S'$ with $S'\subset S$, we have
\begin{align} \label{p1}
\dim(\Psi_t(C)\cap S') &\leq \dim(C\cap S) + \dim S' - \dim S,\\ 
\label{p2}
\dim(\Psi_t(\Sigma C)\cap S') &\leq \dim(\Sigma C\cap S) + \dim S' - \dim S. 
\end{align} 
By (\ref{perv1}) for $S$, (\ref{p1}) implies
\begin{align*}
\dim(\Psi_t(C)\cap S') &< \dim C - \half \codim S +\dim S'-\dim S,\\
&\le  \dim C - \half \codim S',
\end{align*}
so the cycle $\Psi_t(C)$ satisfies the perversity condition (\ref{perv1}) for $S'$. A similar argument applied to (\ref{p2}) shows that $\Psi_t(\Sigma C)$ satisfies the perversity condition (\ref{perv2}) for $S'$. Thus $\varphi_k$ is surjective.

Now we show that $\varphi$ is injective. Suppose the $k$-cycle $C$ is allowable for $\mc S'$, and $C$ is the boundary of a $(k+1)$-chain $D$ that is allowable for $\mc S$. We apply the preceding argument to the chain $D$. There is an arc-analytic $\mc S$-stratum-preserving semialgebraic homeomorphism $\Psi_{t_0}: \proj^r\to \proj^r$ with arc-analytic inverse such that $\Psi_{t_0}(D)$ is allowable for $\mc S'$.

Choose $t_0$ so that $\{t\ ;\ t = st_0,\ 0<s\leq 1\}$ is contained in the open set $U'$ given by Proposition \ref{genpos} (a). Let $B$ be the $t_0$-deformation homology \eqref{defhom} of $C$ with respect to $\Psi_t$.
Then $B$ is an $\mathcal S'$-allowable homology with boundary $C+\Psi_{t_0}(C)$. (Condition \eqref{perv2} for $B$ follows from Proposition \ref{defchain}.)
Thus $C$ is null-homologous, so we have shown that $\varphi_k$ is injective.

To complete the proof of Theorem \ref{refinement}, suppose that $X\subset \proj^s$ is another embedding in a projective space. Let $\mc S$ be a good stratification with respect to the embedding $X\subset \proj^r$. By Corollary \ref{goodrefine} (a) applied to $X\subset \proj^s$, there is a refinement $\mc S'$ of $\mc S$ such that $\mc S'$ is good for the embedding $X\subset \proj^s$. So there's a canonical homomorphism $\varphi_k: \ih _k ^{\mc S'} (X)\to \ih _k ^{\mc S} (X)$, and the preceding argument shows that $\varphi_k$ is an isomorphism.
\end{proof}

The preceding proof gives the following more general result.

\begin{cor}
If $V$ is an open semialgebraic subset of the subvariety $X$ of $\proj^n$, and $\mc S$ is a good stratification of $X$ such that $V$ is a union of strata, the groups $\ih^{c,\mc S}_k(V)$, $\ih^{cl,\mc S}_k(V)$, and $\ih^{c,\mc S}_k(X,X\setminus V)$ are independent of $\mc S$.
\end{cor}

\begin{rem}
If $X$ is a subvariety of $\R^r$, the proof of Theorem \ref{refinement} shows that $\ih^{c,\mc S}_k(X)$ and $\ih^{cl,\mc S}_k(X)$ do not depend on the stratification $\mc S$ or the embedding $X\subset\R^r$. Moreover, if we have embeddings $X\subset\proj^r$ and $X\subset\R^{r'}$, the intersection homology groups $\ih^{\mc S}_k(X)$ defined with respect to these two embeddings are canonically isomorphic.
\end{rem}

The geometric definition of $\ih_k(X)$ implies that these groups have properties analogous to the classical intersection homology groups of Goresky and MacPherson. For example the intersection homology groups are related to ordinary homology and cohomology in the following manner.

\begin{prop}
Let $X$ be an $n$-dimensional subvariety of $\proj^r$. For all $k\geq 0$ there are canonical homomorphisms 
$$
H^{n-k}(X;\Z/2)\to \ih_k(X) \to H_k(X;\Z/2)
$$
such that the composition is the classical Poincar\' e duality homomorphism.
\end{prop}

\begin{proof}
The homomorphism $\ih_k(X) \to H_k(X;\Z/2)$ is induced by the chain map that takes an allowable chain to the corresponding semialgebraic chain, forgetting the perversity conditions. Let $l=n-k$. By Alexander-Lefschetz duality, we identify the cohomology of $X$ with the relative homology of $\proj^r$ modulo the complement of $X$,
$$
H^l(X;\Z/2) \cong H_{r-l}(\proj^r, \proj^r\setminus X;\Z/2).
$$
The Poincar\' e duality map $D_k:H^l(X;\Z/2)\to H_k(X;Z/2)$ (defined as cap product with the fundamental class in $H_n(X;\Z/2)$) can be described geometrically using the above identification (\emph{cf.} \cite{goresky}). If $\gamma\in H^l(X;\Z/2)$, then $\gamma$ can be represented by a relative $(r-l)$-cycle $C$ in $(\proj^r, \proj^r\setminus X)$ that is transverse to $X$. In other words, there are stratifications $\mathcal A$ of $C$ and $\mathcal B$ of $X$ such that, for every pair of strata $A\in\mathcal A$ and $B\in\mathcal B$, we have that $A$ and $B$ are transverse in $\proj^r$. Then $D_k(\gamma)$ is represented by the intersection $C\cap X$. If $\Sigma C$ is a union of strata of $\mathcal A$, then $C\cap X$ is an allowable cycle of $X$. Thus the homomorphism $D_k$ factors through $\ih_k(X)$.
\end{proof}

We conclude this section with some remarks on perversity conditions. Let $\mathcal S$ be a good stratification of the real algebraic subvariety $X$ of $\proj^r$. For simplicity, in this informal discussion we identify a semialgebraic $k$-chain with its support, a closed $k$-dimensional semialgebraic subset of $X$. Following \cite{king}, we define a \emph{loose perversity} to be a sequence of integers $\overline p = (p_0,p_1,p_2,p_3\dots)$ with $p_0 = 0$. We say that a semialgebraic $k$-chain $C$ in $X$ is $\overline p$-\emph{allowable} (with respect to $\mathcal S$) if, for all $i\geq 0$ and all strata $S\in \mathcal S$ of codimension $i$,
\begin{align*}
\dim(C\cap S)&\leq k -i + p_i ,\\
\dim(\partial C\cap S)&\leq (k-1) -i + p_i .
\end{align*}
Given a pair of loose perversities $(\overline p,\overline q)$, we say that the $k$-chain $C$ is $(\overline p,\overline q)$-\emph{allowable} if $C$ is $\overline p$-allowable and the pseudoboundary $\Sigma C$ (a $(k-1)$-chain) is $\overline q$-allowable.

With this notation our definition of an allowable $k$-chain $C$ of $X$ is equivalent to the statement that $C$ is $(\overline p,\overline q)$-allowable, with $\overline p = (0,0,0,1,1,2,2,\dots)$ and $\overline q = (0,0,1,1,2,2,\dots)$. Note that $\overline p$ is the ``upper middle'' perversity $\overline n$ of Goresky-MacPherson (\cite{gm1} \S 5, \cite{gm1a} \S6). (Their perversity sequence starts with $p_2$.)

A useful definition of intersection homology groups $\ih^{\overline p,\overline q}_*(X)$ for real algebraic varieties $X$, defined by a pair of loose perversity conditions $(\overline p,\overline q)$ on semialgebraic chains, should have at least the following two properties: (1) If $X$ is nonsingular, then $\ih^{\overline p,\overline q}_*(X) \cong H_*(X;\Z/2)$ (compact or closed supports), (2) $\ih^{\overline p,\overline q}_*(X)$ is independent of the good stratification $\mathcal S$.

Using general position arguments, we can show that the groups $\ih^{\overline p, \overline q}_*(X)$ satisfy (1) and (2) if the following conditions hold for all $i\geq 0$ (\emph{cf.} \cite{king} Theorem 9, and the proofs of Theorem \ref{refinement} above and Theorem \ref{nashhomology} below):
\begin{align*}
&p_i\leq p_{i+1}\leq p_i + 1,\\
&q_i\leq q_{i+1}\leq q_i + 1,\\
&p_i\leq q_i\leq p_i + 1.
\end{align*}
It remains to be seen which of these groups have useful properties. 

A case of interest is those pairs $(\overline p, \overline q)$ with the above restrictions as well as the following:
 \begin{align*}
&p_1=p_2=0,\\
&q_i = p_i + 1 \ \ \text{for $i$ even, $i>0$}.
\end{align*}
(Our definition of intersection homology satisfies these conditions.) Thus $\overline p$ is a Goresky-MacPherson perversity (\cite{gm1} \S 1.3), and there is no condition on $\Sigma C\cap S$ when $S$ has even codimension. So if all strata have even codimension we have the classical intersection homology groups.
 

\section{Nash approximation and arc-symmetric cycles} 
\label{Nash} 

Let $X\subset \R^r$ and $Y\subset \R^s$ be nonsingular 
real algebraic subvarieties, with $X$ compact.  The map $h:X \to Y$ is \emph{Nash} if it is real analytic and the graph of $h$ is semialgebraic. The map $h$ is Nash if and only if it is semialgebraic and $C^{\infty}$ (\cite{bcr} ch. 8). If $h :X\to X$ is a Nash isomorphism then the image by $h$ of a $k$-dimensional semialgebraic  subset $C$ of $X$ that is arc-symmetric at $x\in C$ (respectively, arc-symmetric in dimension $k$ at $x\in C$) is again semialgebraic and arc-symmetric at $h(x)\in C$ (respectively, arc-symmetric in dimension $k$ at $h(x)\in C$).  

Every compact nonsingular real algebraic variety admits a \emph{submersive Nash deformation of the identity}. More precisely, we have the following theorem.

\begin{thm}\label{nashdef}
Let $X$ be a compact nonsingular real algebraic variety. There is a Nash map $\Theta: U\times X\to X$, where U is a semialgebraic open neighborhood of the origin in $\R^l$, with the following properties:
\begin{enumerate}
\item $\Theta_t:X\to X$ is a Nash isomorphism for all $t\in U$, where $\Theta_t(x) = \Theta(t,x)$.
\item $\Theta_0(x) = x$ for all $x\in X$.
\item For $\Theta^x:U\to X$, the differential $D\Theta^x$ at $t$ is surjective for all $x\in X$, where $\Theta^x(t)= \Theta(t,x)$.
\end{enumerate}
\end{thm}
\begin{proof}
First we observe that there is a $C^\infty$ map $\Psi: T\times X \to X$, where $T$ is an open neighborhood the origin in $\R^l$, with the corresponding properties: (1) $\Psi_t$ is a diffeomorphism for all $t\in T$, (2) $\Psi_0(x) = x$ for all $x\in X$, and (3) $D\Psi^x$ at $t$ is surjective for all $x\in X$. The construction is given in \cite{gm2} (1.3.7): Choose finitely many $C^\infty$ vector fields $V_1,\dots,V_l$ on $X$ such that at each point $x\in X$ the vectors $V_1(x),\dots,V_l(x)$ span the tangent space $T_xX$. Let $\R^l$ be the vector space of formal linear combinations of the $V_i$, let $\Psi_t$ be the time 1 flow of the corresponding vector field, and let $T$ be a small open ball centered at the origin. 

Now let $\Psi':U\times X\to X$ be a Nash approximation of $\Psi$ (\emph{cf.}\ \cite{bcr} 8.9.7, \cite{shiota} I.3.4). More precisely, let $U$ be an open ball centered at the origin in $\R^l$, with $\cl(U)\subset T$. Let $\overline\Psi': cl(U)\times X\to X$ be a Nash approximation of $\Psi | (\cl(U)\times X)$ and let $\Psi' = \overline \Psi' | (U\times X)$. The approximation $\overline\Psi'$ is constructed using a Nash tubular neighborhood of the target $X$ (\cite{bcr} 8.9.5) and the Stone-Weierstrass Theorem (\cite{bcr} 8.8.5).

Since properties (1) and (3) of $\Psi$ are open in $C^\infty(T\times X, X)$, we can assume that $\Psi'$ also has these properties. Consider the function $\tau:X\to U$ such that $\Psi'_{\tau(x)}(x) = x$ for all $x\in X$. The graph of $\tau$ is the inverse image of the diagonal by the map $f:X\times U\to X\times X$, $f(x, t) = (\Psi_t(x), x)$. Thus $\tau$ is a Nash function. Let $h:U\times X\to U\times X$ be the Nash isomorphism $h(t,x) = (t-\tau(x), x)$, and let $\Theta(t,x)=\Psi'(h(t,x))$. Then $\Theta(0,x) = x$ for all $x\in X$, and so $\Theta$ has the desired properties (1), (2), and (3).
\end{proof}

\begin{prop}\label{nashgp}
Let $X$ be a compact nonsingular purely $n$-dimensional real algebraic variety, and let $\Theta: U\times X\to X$ be a submersive Nash deformation of the identity. Let $Z$ and $W$ be semialgebraic subsets of $X$.
There is an open dense semialgebraic subset $U'$ of $U$ such that, for all $t \in U'$, 
\begin{equation*}
\dim(\Theta_t(Z)\cap W)\le \dim Z + \dim W - n
\end{equation*}
\end{prop}
\begin{proof}
The proof is the same the corresponding stratified general position result, Proposition \ref{genpos} (a) above (see \cite{mpp}).
\end{proof}

\begin{thm}\label{nashhomology}
If $X$ is a nonsingular real algebraic variety of pure dimension $n$, then $\ih^c_k(X)\cong H_k(X;\Z/2)$   and $\ih^{cl}_k(X)\cong H^{cl}_k(X;\Z/2)$ for $k=0,\dots,n$.
\end{thm}
 \begin{proof}
 (1) Suppose $X$ is compact. Let $X\subset \R^r$ be an embedding, and let $\mc S$ be a good stratification of $X$ with respect to this embedding. Since the complex of allowable chains of $X$ is a subcomplex of the complex of semialgebraic chains of $X$, for each $k\geq 0$ there is a canonical homomorphism $\varphi_k:\ih^{\mc S}_k(X)\to H_k(X;\Z/2)$. We claim that $\varphi_k$  is an isomorphism. 
 
Let $\Theta:U\times X\to X$ be a submersive Nash deformation of the identity as in Theorem \ref{nashdef}.
 Given $\gamma\in H_k(X;\Z/2)$, let $C$ be a semialgebraic cycle representing $\gamma$.   By Proposition \ref{nashgp},  there is an open dense subset $U'\subset U$ such that for $t\in U'$ we have that  $\Theta_t(C)$ is in general position with $S$ for all $S\in\mc S$. Then $\Theta_t(C)$ is homologous to $C$ and $\Theta_t(C)$ is allowable for $\mc S$. Therefore $\varphi_k$ is surjective.
 
 Now let $\xi\in\ih^{\mc S}_k(X)$, and suppose that $\varphi_k(\xi) = 0$. Let $C$ be an allowable $k$-cycle in $X$ representing $\xi$, and let $D$ be a semialgebraic $(k+1)$-chain in $X$ such that $C=\p D$. As above, there is a perturbation $\Theta_{t_0}$ of the identity of $X$ such that $\Theta_{t_0}$ is a Nash isomorphism and $\Theta_{t_0}(D)$ is in general position with $\{S\ ;\ S\in \mc S\}$. Then $\Theta_{t_0}(D)$ is allowable. Choose $t_0$ so that $\{t\ ;\ t = st_0,\ 0<s\leq 1\}$ is contained in the open set $U'$ given by Proposition \ref{nashgp}. Let $B$ be the $t_0$-deformation homology \eqref{defhom} of $C$ with respect to $\Theta_t$. Then $B$ is an allowable homology from $C$ to $\Theta_{t_0}(C) =\p \Theta_{t_0}(D)$. Therefore $\xi=0$, and so $\varphi_k$ is injective. 
 
 (2) Suppose $X$ is not compact. We embed $X$ as a Zariski open subset of a compact nonsingular variety $\overline X$. Let $\mc S$ be a good stratification of $\overline X$ such that $X$ is a union of strata.
 
First we consider homology with compact supports. Again we have a canonical homomorphism $\varphi_k:\ih^{c,\mc S}_k(X)\to H_k(X;\Z/2)$, and we claim that $\varphi_k$  is an isomorphism. We apply the preceding argument to a submersive Nash deformation of the identity $\Theta:U\times \overline X\to \overline X$. Given $\gamma\in H_k(X;\Z/2)$, let $C$ be a semialgebraic cycle representing $\gamma$.   By Proposition \ref{nashgp},  there is an open dense subset $U'\subset U$ such that for $t\in U'$ we have that  $\Theta_t(C)$ is in general position with $S$ for all $S\in\mc S$ and $\Theta_t(C)$ is contained in the open subset $X$ of $\overline X$. Then $\Theta_t(C)$ is homologous to $C$ and $\Theta_t(C)$ is allowable for $\mc S$. Therefore $\varphi_k$ is surjective.
The proof that $\varphi_k$ is injective also parallels the previous case. If $C$ is allowable and $C =\p D$, we just have to choose $t_0\in U'$ so that $\Theta_{t_0}(D)$ and the $t_0$-deformation homology of $C$ are contained in the open subset $X$ of $\overline X$.
 
Finally we consider homology with closed supports. We show that the canonical homomorphism $\varphi_k:\ih^{cl}_k(X)\to H^{cl}_k(X;\Z/2)$ is an isomorphism. By Proposition \ref{openIH} and Proposition \ref{openH}, this is equivalent to proving that the corresponding  homomorphism 
$$
\varphi'_k:\ih^{c,\mc S}_k(\overline X,\overline X\setminus X)\to H_k(\overline X, \overline X\setminus X;\Z/2)
$$ 
is an isomorphism. Again we use $\Theta:U\times \overline X\to \overline X$, a submersive Nash deformation of the identity. Given $\gamma\in H_k(\overline X, \overline X\setminus X;\Z/2)$, let $C$ be a semialgebraic chain representing $\gamma$. Thus $\p C \subset \overline X\setminus X$.  By Proposition \ref{nashgp},  there is an open dense subset $U'\subset U$ such that for $t\in U'$ we have that  $\Theta_t(C)$ and $\Theta_t(\p C)$ are in general position with $S$ for all $S\in\mc S$.  Choose $t_0\in U'$ so that the $t_0$-deformation homology $B$ of $C$ and the $t_0$-deformation homology $B'$ of $\p C$ are in general position with $S$ for all $S\in\mc S$. Then $\Theta_{t_0}(C) + B'$ is allowable, and $\p B = C + (\Theta_{t_0}(C) + B')$, and so $\p(\Theta_{t_0}(C) + B')= \p C$,  and $C$ is homologous to $\Theta_{t_0}(C) + B'$ in $C_k(\overline X,\overline X\setminus X;\Z/2)$. 
Therefore $\varphi'_k$ is surjective.

Now let $\xi\in \ih^{cl,\mc S}(\overline X,\overline X\setminus X)$, and suppose that $\varphi_k'(\xi) = 0$. Let $C$ be an allowable $k$-chain in $\overline X$ representing $\xi$, so $\p C \subset \overline X\setminus X$. Let $D$ be a $(k+1)$-chain with $\p D = C + C'$, where $C'\subset \overline X\setminus X$. There is a perturbation $\Theta_{t_0}$ of the identity of $\overline X$ such that $\Theta_{t_0}(D)$ and $\Theta_{t_0}(\p D)$ are in general position with $\{S\ ;\ S\in\mc S\}$. Choose $t_0$ so that the $t_0$-deformation homology $B$ of $\p D$ is in general position with $\{S\ ;\ S\in\mc S\}$. Then $\Theta_{t_0}(D)+ B$ is allowable, and $\p(\Theta_{t_0}(D)+B)= C + C'$, so $C$ is homologous to zero in $C^{c,\mc S}_k(\overline X,\overline X\setminus X)$. Thus $\varphi'_k$ is injective.
 \end{proof}

 The preceding proof applies \emph{mutatis mutandis} when $X$ is replaced with $V$, an open semialgebraic subset of an algebraic variety.

\begin{cor}\label{nashhomologycor}
Let $V$ be an open semialgebraic subset of the nonsingular real algebraic variety $X$ of pure dimension $n$. Then $\ih^c_k(V)\cong H_k(V;\Z/2)$   and $\ih^{cl}_k(V)\cong H^{cl}_k(V;\Z/2)$ for $k=0,\dots,n$.\qed
\end{cor}
 
 The following generalization of Thom representability for the homology of a nonsingular variety is due to Kucharz \cite{kucharz} (\emph{cf}.\ \cite{kp}, Theorem 6.1).
 \begin{prop}\label{arcsymrep}
If $X$ is a  nonsingular real algebraic variety, every  homology class  in $H_*(X;\Z/2)$ has a semialgebraic arc-symmetric representative.
\end{prop}
\begin{proof}
Let $\alpha\in H_k(X;\Z/2)$.
By \cite{thom} (Th\' eor\` eme III.2) there is a compact manifold $V$ of dimension $k$ and 
a smooth mapping $f: V\to X$ such that $\alpha = f_* ([V])$, where 
$[V]$ denotes the fundamental class of $V$.  By the Nash-Tognoli
theorem (\cite{bcr} Theorem 14.1.10, \cite{bk} Theorem 1.2) we may suppose that $V$ is a nonsingular affine algebraic variety and, by Nash approximation (\cite{bcr} Corollary 8.9.7), that $f$ is a Nash mapping. 

If $V_1,\dots,V_r$ are the connected components of $V$, then $\alpha = f_*([V]) = f_*([V_1]) + \cdots +f_*([V_r])$. Let $V'$ be a connected component of $V$.  We may assume that $f'=f|V'$ is generically finite over $f(V')$; otherwise
$f_*([V'])=0$.  By \cite{mccpar1} (Proposition 5.1) the fibers 
of $f'$ are of generically constant parity over $f(V')$, and we may assume that the generic fiber has an odd number of points; otherwise again $f_*([V'])=0$.  It follows from \cite{kurdyka1} and \cite {mccpar1} (\S 5) that $f(V')$ is arc-symmetric in dimension $k$, or equivalently in terms of  arc-symmetric closure that 
$\dim (\as {f(V')} \setminus f(V')<k$.  And 
$f_*([V'])$ is represented by $f(V')$ or by its arc-symmetric
closure.  
\end{proof}

We will need the following relative version of the preceding theorem.

\begin{prop}\label{relative}
Let $Y\subset\R^r$ be a nonsingular real algebraic variety of pure dimension $n$. Let $X\subset Y$ be a compact $n$-dimensional submanifold of $Y$ such that the boundary $\p X$ is a subvariety of $Y$. Let $U$ be a closed neighborhood of $\p X$ in $X$. Then any class in $H_*(X,\p X;\Z/2)$ has a semialgebraic representative that is arc-symmetric in $X\setminus U$.
\end{prop}
\begin{proof}
The quotient $X/\p X$ is a compact semialgebraic set. By \cite{thom} any class in $H_k(X/\p X;\Z/2)$ can be represented by the image of the fundamental class of a compact smooth manifold $V$ of dimension $k$ by a continuous map $f:V\to X/\p X$. We may suppose that $V$ is real algebraic and that $f$ is semialgebraic and, moreover, $C^\infty$ over $X\setminus U$. The rest of the proof is similar to the preceding argument.
\end{proof}


\section{Small resolutions} 
\label{Small} 

A \emph{small resolution} of a variety $X$ is a resolution of singularities $\pi: \widetilde X\to X$  such that  for all $i\ge 1$, 
\begin{equation}\label{small1}
\dim \{ x\in X\ ;\ \dim \pi \inv (x) \ge i \} < n-2i. 
\end{equation}
If $\mc S$ is a stratification of $X$ such that the fibres of 
$\pi$ are of constant dimension over each stratum of $\mc S$,  
then \eqref{small1} is equivalent to 
\begin{equation}\label{small2}
\dim \pi \inv (x) < \half \codim S , 
\end{equation}
for every stratum $S$ such that $\dim S < \dim X$, and 
every $x\in S$.  

\begin{thm}\label{smallres} Let $\pi: \widetilde X\to X$ be a small resolution of the  real algebraic variety $X$. For all $k\geq 0$ there are canonical isomorphisms
$$
\pi_k: H_k(\widetilde X;\Z/2) \overset{\approx}{\longrightarrow} \ih^c_k(X).
$$
\end{thm}

\begin{cor}\label{goresky}
Let $\pi^i : \widetilde X_i \to X$, $i=1,2$, be two small resolutions 
of $X$.  Then  $H_*(\widetilde X_1;\Z_2)$ and 
$H_*(\widetilde X_2;\Z_2)$  are canonically isomorphic.  
\qed \end{cor}

This corollary answers positively a question posed in 1980 by Goresky and MacPherson \cite{gm3} (end of \S E, p.\ 228). Totaro announced a proof in 2002 via crepant resolutions \cite{totaro} (Cor.\ 3.4, p.\ 537). Van Hamel published a proof in 2003 using Galois equivariant cohomology \cite{vanhamel} (Cor.\ 4.10, p.\ 1410). 

\begin{rem} A consequence of Theorem \ref{smallres} is that $\ih^c_k(X)$ is not a topological invariant. It is easy to find homeomorphic plane algebraic curves with small resolutions that have different homology. An example is given by Goresky and MacPherson \cite{gm3} (p.\ 227). Another example is given by the family of curves $X_t =\{(x^2+y^2)^2 - x^2 +ty^2\}$.  For $t\geq 0$ all the curves $X_t$ are homeomorphic. For $t>0$, the curve $X_t$ is an irreducible curve with an ordinary node at the origin, and a small resolution is homeomorphic to a circle. But $X_0$ is reducible, with a tacnode at the origin, and a small resolution is homeomorphic to the disjoint union of two circles.
\end{rem}

\begin{proof}[Proof of Theorem \ref{smallres}]
Let $Y\subset X$ be an algebraic subvariety of $X$, $\dim Y<\dim X$, such that for $E= \pi \inv (Y)$, 
$\pi|(\widetilde X\setminus E): {\widetilde X\setminus E}\to X\setminus Y$ is an isomorphism.  Let $\mc A\to \mc B$ be a stratification of the map $\pi$, where $\mc A$, $\mc B$ are semialgebraic Whitney stratifications of $\widetilde X$, $X$, respectively, such that $Y$ a union of strata of $\mc B$, and $E$ is a union of strata of $\mc A$ (\emph{cf}.\ \cite{gm2}, \S 1.7). Then the fibers of $\pi$ are of constant dimension over each stratum of $\mc B$. By Proposition \ref{goodrefinesemialg} there is a good stratification $\mc S$ of $X$ (with respect to an embedding $X\subset\proj^r$) that refines $\mc B$.

The idea of the proof is the following. If a $k$-chain of $\widetilde X$ is in general position with the sets $\pi\inv(S)$ for all $S\in \mc S$, then its image in $X$ is allowable with respect to $\mc S$. The inverse chain map takes an allowable $k$-chain of $X$ to its strict transform in $\widetilde X$.

\begin{lem}\label{bi}
For each $S\in \mc S$, let $\widetilde S = \pi\inv(S)$. Let $\widetilde C$ be a $k$-dimensional semialgebraic 
subset of $\widetilde X$ such that $\widetilde C$ and $\Sigma\widetilde C$ are in general position with 
respect to all the sets $\widetilde S$.  Then $C=\pi (\widetilde C)$ satisfies the perversity conditions \eqref{perv1} and \eqref{perv2}.  
\end{lem}
\begin{proof}
By construction of the stratification $\mc S$, for each $S\in\mc S$ we have that $\dim \pi\inv(x)$ is constant for $x\in S$. Thus for $x\in S$ we have $\dim \widetilde S = \dim S + \dim  \pi\inv(x)$, and $\dim\pi\inv(x) < \half \codim S$ by (\ref{small2}). It follows that 
\begin{equation*}
\codim \widetilde S > \half \codim S. 
\end{equation*} 
The general position hypothesis for $\widetilde C$ is
\begin{equation*}
\dim(\widetilde C\cap\widetilde S) \leq k - \codim\widetilde S. 
\end{equation*}
Therefore 
\begin{equation*}
\dim(C\cap S)\leq\dim(\widetilde C\cap\widetilde S)\leq k-\codim\widetilde S < k-\half\codim S,
\end{equation*}
so we have the perversity condition \eqref{perv1}.

Note that $\pi$ gives a one-to-one correspondence between the real analytic curves in $\widetilde X$ not contained entirely in the exceptional set $E\subset \widetilde X$ and the real analytic curves in $X$ not contained entirely in $Y$. 

 Therefore 
\begin{equation}\label{arcimage}
\Sigma C = \pi (\Sigma \widetilde C).
\end{equation}
Thus \eqref{perv2} for $\codim S$ odd follows by the above argument applied to $\Sigma \widetilde C$.  For $\codim S$ even \eqref{perv2} is redundant and follows from \eqref{perv1}.  This completes the proof of Lemma \ref{bi}.
\end{proof}

We use Lemma \ref{bi} to define a homomorphism  
\begin{equation*}
\pi_k: H_k(\widetilde X,\Z_2) \to \ih_k^{c,\mc S} (X). 
\end{equation*}
Embed $\widetilde X$ as a Zariski open subset of a compact nonsingular variety $\widetilde Y$.
If $\widetilde\gamma \in H_k(\widetilde X,\Z_2)$, by Proposition \ref{nashgp} applied to $\widetilde Y$ there is a  $k$-cycle $\widetilde C$ representing $\widetilde\gamma$ such that $\widetilde C$ and $\Sigma\widetilde C$ are in general position with respect to $\{\pi\inv(S)\ ;\ S\in\mc S \}$, so by Lemma \ref{bi} we can set 
$
\pi_k(\widetilde\gamma) = [\pi(\widetilde C)].
$
 
Again by Proposition \ref{nashgp} for $\widetilde Y$ and the deformation homology construction (as in the proof of Theorem \ref{nashhomology}), two homologous semialgebraic cycles that are in general position with all  $\pi\inv (S)$ are homologous by a semialgebraic chain in general position with all $\pi\inv (S)$. Therefore $\pi_k$ is well-defined.

To prove that $\pi_k$ is an isomorphism we construct its inverse.  For this we define the \emph{strict transform} $s (C) \subset \widetilde X$ of 
an allowable $k$-chain $C\subset X$ by  
$$
s(C) = \cl(\pi \inv (C) \setminus E),  
$$
which has dimension $k$.

\begin{lem}\label{stricttrans}
\hfill
\begin{enumerate}
\item [(i)]
$\Sigma(s(C))\subset s(C) \cap \pi \inv (\Sigma C)$;
\item [(ii)]
$\p s(C) \sim s(\p C)$.
\end{enumerate}
\end{lem}

\begin{proof}
By assumption $\dim (C\cap Y) <\dim C$ and hence $C$ equals 
$\pi (s(C))$ up to a set of dimension $<k$.  Thus 
(i) follows from \eqref{arcimage} with $\widetilde C$ replaced by 
$s(C)$.

To show (ii), note that  $\p s(C)$ and $s(\p C)$ 
coincide outside $E$, and hence $s(\p C) \subset 
\p s(C)$.   To show the opposite inclusion 
we note first that $\p s(C) \subset \Sigma (s(C)) 
\subset s(C) \cap \pi \inv (\Sigma (C))$ by \eqref{sigma} and (i).  Thus it suffices to show 
that  $E \cap \pi \inv (\Sigma (C))$ is  of dimension smaller than $k-1$.  
This follows from \eqref{perv2} and \eqref{small2}: 
For any stratum $S$ such that 
$\dim S< \dim X$, $x\in S$, 
\begin{align*}
\dim \pi \inv (S\cap \Sigma (C)) &= \dim \pi \inv (x) 
+ \dim (S\cap \Sigma (C))\\
&< \half \codim S + (k - \half \codim S - 1)\\ &= k-1 .
\end{align*}
\end{proof}

In particular if $C$ is a cycle so is $s(C)$.  
Lemma \ref{stricttrans} shows that $s$ induces a homomorphism 
\begin{equation*}
s_k:  \ih^c_k (X) \to H_k(\widetilde X,\Z_2) . 
\end{equation*}      
To show that $s_k$ and $\pi _k$ are inverse to each other, we note that if $\widetilde C$ is in general position with 
respect to all $\pi \inv (S)$, then $s(\pi(\widetilde C))\sim \widetilde C$ (actually $\supp 
s(\pi(\widetilde C))= \supp \widetilde C$).  So if $\widetilde\gamma$ is the homology class of $\widetilde C$, then $s_k(\pi_k(\widetilde\gamma)) = \widetilde\gamma$ .
On the other hand, let $\gamma\in\ih_k(X)$ be the class of the allowable cycle $C$.  Then $s(C)$ represents $s_k(\gamma)$.  By Proposition \ref{nashgp} for $\widetilde Y\supset \widetilde X$, there is a perturbation $\Theta_{t_0}$ of the identity of $X$ such that $\Theta_{t_0}(s(C))$ is in general position with all the sets $\pi\inv(S)$ for $S\in \mc S$. The cycle $\pi(\Theta_{t_0}(s(C)))$ represents $\pi_k(s_k(\gamma))$. Choose $t_0$ so that $\{t\ ;\ t=st_0,\ 0<s\leq t\}$ is contained in the open set $U'$ given by Proposition \ref{nashgp} for $\widetilde Y$, and let $B$ be the $t_0$-deformation homology \eqref{defhom} of $s(C)$ with respect to $\Theta_t$. Then $\partial B = s(C) + \Theta_{t_0}(s(C))$.  It follows that $\pi(s(C))$ is homologous to $\pi(\Theta_{t_0}(s(C)))$ by an allowable homology. Now $\pi(s(C))\sim C$, so $C$  represents  $\pi_k(s_k(\gamma))$.   

This completes the proof of Theorem \ref{smallres}.
\end{proof}


\section{Intersection pairing} 
\label{intersection}

The intersection number of intersection homology classes of complementary dimension is defined by transverse intersection of allowable cycles.  

\begin{thm}\label{duality}
Let $X$ be a real algebraic variety of dimension $n$. For all pairs of non-negative integers $k_1$, $k_2$, with $k_1+k_2 = n$, there is a bilinear pairing
$$
I:\ih^c_{k_1}(X)\times \ih^{cl}_{k_2}(X)\to \Z/2 .
$$
\end{thm}
\begin{proof}
Let $\gamma_1\in\ih^c_{k_i}(X)$ and $\gamma_2\in\ih^{cl}_{k_i}(X)$. We define the intersection number $I(\gamma_1,\gamma_2)\in \Z/2$ as follows. Embed $X$ as a subvariety of affine space, and let $\mc S$ be a good stratification of $X$ with respect to this embedding. Let $C_i$ be $\mc S$-allowable $k_i$-cycles representing $\gamma_i$, $i=1,2$, such that $C_1$ has compact support. Let $\mc A_i$ be a semialgebraic stratification of $C_i$ such that $\Sigma C_i$ is a union of strata of $\mc A_i$, and $\mc A_i$ is a substratified object of $\mc S$, $i=1,2$.
By Proposition \ref{genpos} there exists a stratified semialgebraic and arc-analytic isomorphism $\Psi_t:X\to X$ such that $\Psi_t(C_1,\mc A_1)$ is transverse to $(C_2,\mc A_2)$ in $(X,\mc S)$. Let $C_1' = \Psi_t(C_1)$. Transversality implies that for all strata $S\in\mc S$, we have
$$
\dim(C_1'\cap C_2\cap S) \leq
\dim(C_1'\cap S) + \dim(C_2\cap S) - \dim S .
$$
The perversity conditions \eqref{perv1} for $C_1$ and $C_2$ imply that $C_1'\cap C_2\cap S =\emptyset$ for all $S$ such that $\codim S > 0$, and  $\dim(C_1'\cap C_2\cap S)= 0$ if $\codim S = 0$. (The perversity conditions \eqref{perv2} imply that $\Sigma C_1'\cap C_2\cap S= \emptyset$ and $ C_1'\cap \Sigma C_2\cap S= \emptyset$ for all $S\in \mc S$.) Thus $C_1'$ and $C_2$ intersect in a finite number $m$ of points that lie in the top strata of $X$ (finite because $C_1'$ is compact), and each intersection point is the transverse intersection of top strata of $C_1'$ and $C_2$. Let $I(C_1',C_2) = m$. We set 
$$
I(\gamma_1,\gamma_2)= I(C_1',C_2) \pmod 2.
$$
To prove that $I(\gamma_1,\gamma_2)$ is independent of all the above choices ($\mc S$, $C_i$, $A_i$, $\Psi_t$), it suffices to prove the following lemma.
\begin{lem}
If there is a compactly supported allowable $(k_1+1)$-chain $D$ with $\p D = C_1' $, then $I(C_1', C_2) \equiv 0 \pmod 2$. 
\end{lem}
\begin{proof}
Let $\mc E$ be a semialgebraic stratification of $D$ compatible with $\Psi_t(\mc A_1)$ such that $\Sigma D$ is a union of strata of $\mc E$.
By Proposition \ref{genpos} (b) there exists a stratified semialgebraic arc-analytic isomorphism $\Psi'_{t_0}$ of $X$ such that $\Psi'_{t_0}(D,\mc E)$ is transverse to $(C_2,\mc A_2)$. 
Choose $t_0$ so that $\{t\ ;\ t = st_0,\ 0<s\leq 1\}$ is contained in the open set $U'$ given by Proposition \ref{genpos} (b). Let $B$ be the $t_0$-deformation homology \eqref{defhom} of $C'_1$ with respect to $\Psi_t$.
Then $B$ is an $\mathcal S$-allowable homology from $C'_1$ to $\Psi'_{t_0}(\partial D)=\partial \Psi'_{t_0}(D)$, with $B$ transverse to $C_2$. (Condition \eqref{perv2} for $B$ follows from Proposition \ref{defchain}.)

So we have an allowable $(k_1+1)$-chain $D'= B+\Psi'_{t_0}(D)$ such that $\p D' = C_1'$, and $D'$ is in transverse to $C_2$.  Therefore for all $S\in \mc S$ we have
\begin{equation}\label{homologygp}
\dim (D'\cap C_2\cap S) \le \dim (D' \cap S) + \dim (C_2 \cap S) - \dim S . 
\end{equation}
If $\codim S = 2l$ then \eqref{perv1} gives $\dim (D'\cap S)\leq k_1-l$ and $\dim(C_2\cap S)\leq k_2-l-1$, so \eqref{homologygp} gives $\dim (D'\cap C_2\cap S)\leq -1$, and $D'\cap C_2\cap S=\emptyset$. On the other hand, if $\codim S = 2l+1$ then again \eqref{perv1} gives $\dim (D'\cap S)\leq k_1-l$ and $\dim(C_2\cap S)\leq k_2-l-1$, so \eqref{homologygp} implies $\dim (D'\cap C_2\cap S)\leq 0$. 

For $\codim S$ even, $\Sigma(D'\cap C_2\cap S) \subset (D'\cap C_2\cap S)=\emptyset$. Suppose $\codim S = 2l+1$.
Now $\Sigma(D'\cap C_2)\subset (\Sigma D'\cap C_2)\cup(D'\cap\Sigma C_2)$, and \eqref{perv2} gives
\begin{align*}
\dim(\Sigma D'\cap C_2\cap S) &\le
\dim(\Sigma D'\cap S)+\dim(C_2\cap S) -\dim S\\
\dim(D'\cap \Sigma C_2\cap S) &\le
\dim(D'\cap S)+\dim(\Sigma C_2\cap S) -\dim S
\end{align*}
These equations and the condition \eqref{perv2} imply that $\Sigma D'\cap C_2\cap S=\emptyset$ and $D'\cap \Sigma C_2\cap S=\emptyset$. Therefore $\Sigma(D'\cap C_2\cap S)=\emptyset$.

We conclude that, although the compactly supported semialgebraic $1$-chain $D'\cap C_2$ may intersect the singular locus of $X$ in a finite number of points, it is arc-symmetric near these intersection points. Therefore $D'\cap C_2$ is an analytic curve near the codimension $1$ skeleton of $\mc S$, so the boundary of $D'\cap C_2$ lies in the union $X'$ of the top strata of $\mc S$. Since the stratified sets $D'\cap X'$ and $C_2\cap X'$ are transverse in $X'$, it follows that 
\begin{align*}
\p[(D'\cap X')\cap (C_2 \cap X') ] &= 
\p(D'\cap X')\cap (C_2\cap X') + (D'\cap X')\cap \p (C_2\cap X')\\
&= \p D' \cap C_2 \cap X'\\
&= C_1'\cap C_2.
\end{align*}
Therefore $\p(D'\cap C_2) = C_1'\cap C_2$, so $I(C_1'\cap C_2) =0$. This proves the lemma.
\end{proof}
 
Therefore $I(\gamma_1,\gamma_2)$ is well-defined.
 Next we show that the intersection pairing is bilinear. In other words, if $\gamma_1, \gamma_2\in \ih_{k}(X)$ and $\gamma_3\in\ih_{l}(X)$, with $k+l= n$, then
 \begin{equation*}
 I(\gamma_1+\gamma_2,\gamma_3)
 = I(\gamma_1,\gamma_3)+ I(\gamma_2,\gamma_3).
  \end{equation*}
  Let $C_i$ be an allowable $k$-cycle representing $\gamma_i$, $i=1,2,3$, such that $C_1$ and $C_2$ are compactly supported and transverse to $C_3$. We can also assume that $C_1$ and $C_2$ are transverse. It follows that (if $k<n$) $C_1 +C_2 = C_1 \cup C_2$, and 
 \begin{equation*}
(C_1\cup C_2)\cap C_3 =  (C_1\cap C_3) \sqcup (C_2\cap C_3). 
 \end{equation*}
 Therefore
 \begin{equation*}
 I(C_1 + C_2, C_3) = I(C_1,C_3)+ I(C_2,C_3).
 \end{equation*}
This completes the proof of the theorem.
\end{proof}


\section{Isolated Singularities} 
\label{Isolated} 
\begin{thm}
If $X$ is a purely $n$-dimensional real algebraic variety with isolated singularities, the intersection homology intersection pairing is nonsingular. In other words, for $k=0,\dots,n$ the duality homomorphism
\begin{align*}
D_k: \ih^c_k(X)&\to (\ih^{cl}_{n-k}(X))^*\\
\langle D_k(\alpha), \beta\rangle &= I(\alpha,\beta)
\end{align*}
is an isomorphism.
\end{thm}
\begin{proof}
First we prove a local version of the theorem at a singular point, and then we obtain the global version by a Mayer-Vietoris argument.

Suppose that $X$ is a subvariety of $\R^r$. Let $\mc S$ be a good stratification of $X$; the $0$-dimensional strata include the singular points of $X$. 

Let $x_0$ be a singular point of $X$. For $\epsilon>0$ let $B(x_0,\epsilon)$ be the closed ball in $\R^r$ of radius $\epsilon$ about $x_0$, and let $S(x_0, \epsilon)$ be the boundary of $B(x_0,\epsilon)$. Choose $\epsilon$ so small that $S(x_0, \delta)$ is transverse to $\mc S$ for all $\delta<\epsilon$, and $X\cap B(x_0,\epsilon)$ is semialgebraically homeomorphic to the cone on $X\cap S(x_0, \epsilon)$. The \emph{link} $L(x_0) = X\cap S(x_0, \epsilon)$ is a compact nonsingular real algebraic variety of dimension $n-1$. Let $M(x_0) = X\cap B(x_0,\epsilon)$, and let $N(x_0) = M(x_0)\setminus L(x_0)$. 

For the open semialgebraic subset $N(x_0)\subset X$, we  define $\ih^c_*(N(x_0))$ and $\ih^{cl}_*(N(x_0))$ with respect to the fixed stratification $\mc S$. 

\begin{prop}\label{localduality}
$D_k: \ih^c_k(N(x_0))\to (\ih^{cl}_{n-k}(N(x_0)))^*$ is an isomorphism for all $k$.
\end{prop}
\begin{proof}
The proof depends on the parity of the dimension of $X$. Let $N=N(x_0)$ and $L=L(x_0)$. To compute $\ih^c_k(N)$ and $\ih^{cl}_k(N)$, the crucial condition is whether a chain is allowed to meet the $0$-stratum $\{x_0\}$. We will repeatedly use general position (Proposition \ref{nashgp} for a nonsingular compactification $\overline X\supset X$, and the deformation homology construction \eqref{defhom}) and Corollary \ref{nashhomologycor} for the open semialgebraic subset $N'= N\setminus \{x_0\}$ of $X$. We write $H_*( - )$ (resp. $H^{cl}_*( - )$) for $H_*(\ -\ ;Z/2)$ (resp. $H^{cl}_*(\ -\ ;Z/2)$) .

\vskip.1in
(a) Even dimensional case (\emph{cf.} \cite{gm1a} \S2.4). Let $n = 2m$. Consider the perversity condition \eqref{perv1} for $S=\{x_0\}$. (When $S$ has even codimension \eqref{perv2} is a consequence of \eqref{perv1}.) For a $k$-chain $C$, \eqref{perv1} gives
\begin{equation}\label{perv1even}
\dim(C\cap\{x_0\})\leq k - (m+1)
\end{equation}
We will show that \eqref{perv1even} has the following consequences:
\begin{enumerate}
\item If $k\geq m+1$ then $\ih^c_k(N) = 0$ and $\ih^{cl}_k(N) \cong H_{k-1}(L)$.
\item If $k=m$ then $\ih^c_k(N) = 0$ and $\ih^{cl}_k(N) =0$.
\item If $k\le m-1$ then $\ih^c_k(N) \cong H_k(L)$ and $\ih^{cl}_k(N) =0$.
\end{enumerate}

\vskip.1in
First we consider compact supports ($n= 2m$). 

(1)(2) If $k\ge m$ then \eqref{perv1even} says $(k+1)$-homologies are allowed to hit the point $x_0$. If the allowable $k$-cycle $C$ bounds a semialgebraic $(k+1)$-chain $B$, then $C\cap N'$ bounds an allowable chain $B'$ in $N'$ by general position, so $C=\p B''$, where $B''$ is the closure of $B'$ in $N$. Thus the homomorphism $\ih^c_k(N)\to H_k(N)$ is injective. But $H_k(N) = 0$ since $N$ is a semialgebraic cone.

(3) If $k\le m-1$ then \eqref{perv1even} says that both $k$-cycles and $(k+1)$-homologies must miss the point $x_0$. Thus $\ih^c_k(N) = \ih^c_k(N')\cong H_k(N')$ by Corollary \ref{nashhomologycor}, and $H_k(N')\cong H_k(L)$ since $N'$ is semialgebraically homeomorphic to $L\times (0,1)$.

\vskip.1in
Next we consider closed supports ($n= 2m$).

(1) If $k\ge m+1$ then \eqref{perv1even} says $k$-cycles and $(k+1)$-homologies are allowed to hit the point $x_0$.
Thus the homomorphism $\ih^{cl}_k(N')\to \ih^{cl}_k(N)$ that takes a cycle in $N'$ to its closure in $N$ is an isomorphism. By Corollary \ref{nashhomologycor} we have $\ih^{cl}_k(N')\cong H^{cl}_k(N')$, and $H^{cl}_k(N')\cong H_{k-1}(L)$ since $N'$ is semialgebraically homeomorphic to $L\times (0,1)$.

(2)(3) If $k\le m$ then \eqref{perv1even} says $k$-cycles must miss the point $x_0$. So the homomorphism $F_k\to \ih^{cl}_k(N)$ is surjective, where $F_k\subset \ih^{cl}_k(N')$ is the subgroup generated by allowable cycles $C$ such that $x_0\notin \cl(C)$. By general position $F_k\cong G_k$, where $G_k\subset H^{cl}_k(N')$ is the subgroup generated by cycles $C$ such that $x_0\notin \cl(C)$. But $G_k=0$, since $N'\cong L\times (0,1)$, and if $x_0\notin \cl(C)$ then $C\subset L\times [\epsilon, 1)$ for some $\epsilon >0$, and $H_k^{cl}(L\times [\epsilon, 1)) =0$.

\vskip.1in
Since $L$ is a compact nonsingular variety of dimension $n-1$, we have by Poincar\' e Duality  $H_k(L)\cong H_{n-k-1}(L)$, so by (1)(2)(3) above we have $\ih^c_k(N)\cong \ih^{cl}_{n-k}(N)$ for all $k\ge 0$. (Both groups are zero if $k\ge m$.) It remains to show that the duality map $D_k$ is an isomorphism for $k\le m-1$. We claim that the following diagram commutes, where the top arrow is the duality map for the manifold $L$:
\begin{equation*}
\begin{CD}  
H_k (L)  
@>\simeq>> (H_{n-k-1} (L))^*  \\ 
@VV\simeq V @VV\simeq V \\ 
\ih^c_k(N)  
@>D_k>> (\ih^{cl}_{n-k}(N))^*\\
 \end{CD}
 \end{equation*}
The commutativity of this diagram follows the statement that, for all $\alpha\in H_k(L)$ and $\beta\in H_{n-k-1}(L)$, 
\begin{equation}\label{int1}
I(\alpha,\beta) = I(\varphi(\alpha),\psi(\beta)),
\end{equation}
where $\varphi:H_k(L)\to \ih^c_k(N)$ and 
$\psi: H_{n-k-1}(L)\to \ih^{cl}_{n-k}(N)$ are the isomorphisms constructed above ($k\le m-1$). Let 
\begin{align*}
&i:H_k(L)\to H^c_k(L\times (0,1)),\ \ i[C] = [C\times\{\half\}],\\
&j:H_{n-k-1}(L)\to H^{cl}_{n-k}(L\times (0,1)),\ \
j[C] = [C\times (0,1)].
\end{align*}
Then
\begin{equation}\label{int2}
I(\alpha,\beta) = I(i(\alpha),j(\beta)),
\end{equation}
and \eqref{int2} implies \eqref{int1}.

\vskip.1in
(b) Odd dimensional case. Let $n = 2m+1$. Consider the perversity conditions \eqref{perv1} and \eqref{perv2} for $S=\{x_0\}$. For a $k$-chain $C$, these conditions give
\begin{align}
\dim(C\cap\{x_0\})\leq k - (m+1) \label{perv1odd}\\
\dim(\Sigma C\cap \{x_0\})\leq k - (m+2) \label{perv2odd}
\end{align}
We will show that \eqref{perv1odd} and \eqref{perv2odd} have the following consequences:
\begin{enumerate}
\item If $k\geq m+2$ then $\ih^c_k(N) = 0$ and $\ih^{cl}_k(N) \cong H_{k-1}(L)$.
\item $\ih^c_{m+1}(N) = 0$, and $\ih^{cl}_{m+1}(N)$ is isomorphic to a subgroup of $H_m(L)$.
\item $\ih^c_m(N)$ is isomorphic to a quotient group of $H_m(L)$, and $\ih^{cl}_m(N) = 0$.
\item If $k\le m-1$ then $\ih^c_k(N) \cong H_k(L)$ and $\ih^{cl}_k(N) =0$.
\end{enumerate}

\vskip.1in
\noindent
In the following arguments, $C$ denotes an allowable $k$-cycle and $B$ an allowable $(k+1)$-chain.

\vskip.1in
First we consider compact supports ($n= 2m+1$).

(1)(2) If $k\geq m+2$ then $B$ and $\Sigma B$ are allowed to hit $x_0$, so we can repeat the above argument for $n$ even to conclude $\ih^c_k(N) = 0$.
 
(3) If $k = m$ then $C$ misses $x_0$ and $B$ may hit $x_0$, but $\Sigma B$ misses $x_0$. So the inclusion homomorphism $\ih^c_m(N')\to \ih^c_m(N)$ is surjective, and the kernel consists of classes represented by cycles $C$ in $N'$ that bound chains $B$ in $N$ such that $B$ is arc-symmetric in dimension $m$ at $x_0$. By Corollary \ref{nashhomologycor} we have $\ih^c_m(N')\cong H_m(L)$. Thus $\ih^c_m(N)$ is isomorphic to a quotient group of $H_m(L)$.

(4) If $k \leq m-1$ then $C$ and $B$ miss $x_0$, so we repeat the above argument for $n$ even to conclude that $\ih^c_k(N)\cong H_k(L)$.

\vskip.1in
Next we consider closed supports ($n= 2m+1$).

(1) If $k\geq m+2$ then $C$, $\Sigma C$, $B$, and $\Sigma B$ are allowed to hit $x_0$, so we can repeat the above argument for $n$ even to conclude that $\ih^{cl}_k(N)\cong H_{k-1}(L)$.

(2) If $k = m+1$ then $C$ can hit $x_0$ but $\Sigma C$ must miss $x_0$, and $B$, $\Sigma B$ are allowed to hit $x_0$. So the homomorphism $\ih^{cl}_{m+1}(N)\to \ih^{cl}_{m+1}(N')$ is injective, and the image consists of classes represented by cycles $C$ in $N'$ such that the closure of $C$ in $N$ is arc-symmetric in dimension $m+1$ at $x_0$. By Corollary \ref{nashhomologycor} we have $\ih^{cl}_{m+1}(N')\cong H_m(L)$. Thus $\ih^{cl}_{m+1}(N)$ is isomorphic to a subgroup of $H_m(L)$.

(3)(4) If $k \leq m$ then $C$ misses $x_0$, and we repeat the above argument for $n$ even to conclude that $\ih^{cl}_k(N)=0$.

\vskip.1in
Now $L$ is a compact nonsingular variety of dimension $n-1=2m$. We have by Poincar\' e Duality $H_k(L) \cong H_{n-k-1}(L)$ for all $k\geq 0$, so by (1)(2)(3)(4) we have $\ih^c_k(N)\cong \ih^{cl}_{n-k-1}(N)$ for $k\neq m$. (If $k\leq m-1$ then $\ih^c_k(N)\cong H_k(L)\cong H_{n-k-1}(L)\cong \ih^{cl}_{n-k}(N)$, and if $k\geq m+1$ then $\ih^c_k(N)=0= \ih^{cl}_{n-k}(N)$.) We see that the duality map $D_k$ is an isomorphism for $k\leq m-1$ by the same argument as above for the even dimensional case. 

It remains to show that 
$D_m: \ih^c_m(N)\to (\ih^{cl}_{m+1}(N))^*$
is an isomorphism. Recall  that $M=X\cap B(x_0,\epsilon)$, with $L=\p M$ and $N= M\setminus \p M$. We have that the group $ \ih^c_m(N)$ is generated by cycles in $N'$ with compact support. They are divided by boundaries of chains arc-symmetric in a neighborhood of $x_0$. We have an epimorphism
$$
H_m(\p M)=H_m(L)\cong H_m(N')\to \ih^c_m(N).
$$
The group $\ih^{cl}_{m+1}(N)$ is generated by cycles with closed support that are arc-symmetric in a neighborhood of $x_0$. They are divided by arbitrary homologies with closed support. We have a monomorphism
$$
\ih^{cl}_{m+1}(N)\to H^{cl}_{m+1}(N')\cong H_{m+1}(M,\p M).
$$
Let $\pi: \widetilde M\to M$ be a resolution of the singular point $x_0$, with $E=\pi\inv(x_0)$ the exceptional divisor. Thus $\pi|(\widetilde M\setminus E): \widetilde M\setminus E\to M\setminus \{x_0\}$ is an isomorphism, and $\widetilde  M$ is a manifold with boundary $\p \widetilde M \cong \p M=L$. 
\begin{lem}
Consider the diagram
\begin{equation*}\begin{CD} 
H_{m+1}(\widetilde M) @>>> 
H_{m+1} (\widetilde M,\p \widetilde M)  
@>\tilde \p_{m+1}>> H_{m} (\p \widetilde M) 
@>\widetilde i_m>> H_m(\widetilde M) \\ 
@. @VV\pi_{m+1}V @V\simeq VV \\ 
0 @>>>  H_{m+1} ( M,\p M)  
@>\p_{m+1}>\simeq> H_{m} (\p  M) @>>> 0\\
 \end{CD}\end{equation*}
 \begin{align*}
 &\ih^{cl}_{m+1}(N)\cong \Image(\pi_{m+1})\cong\Image (\tilde \p_{m+1})\\
 &\ih^c_m(N)\cong H_m(\p \widetilde M)/\Image (\tilde \p_{m+1})=\Coker(\tilde \p_{m+1})
 \end{align*}
 \end{lem}
\begin{proof}
Let $C$ be an allowable $(m+1)$-cycle with closed support in $N$.  Since $C$ is arc-symmetric in dimension $m+1$ at $x_0$, the strict transform 
$$
s(C) = \cl (\pi \inv (C) \setminus E) 
$$
is an $(m+1)$-cycle with closed support in 
$\widetilde M \setminus \p \widetilde M$.  The strict transform
represents a cycle in $H^{cl}_{m+1} (\widetilde M \setminus 
\p \widetilde M)  \cong 
H_{m+1} (\widetilde M,\p \widetilde M)$. 
  Thus 
$\ih^{cl}_{m+1} (N )$ is contained in the image of $\pi_{m+1}$.  
On the other hand by Proposition \ref{relative} 
each cycle in $H_{m+1} (\widetilde M,\p \widetilde M) 
$ has a representative $\widetilde C$ that is arc-symmetric in a neighborhood  of $E$ and that is in general position with $E$.  The image $\pi_{m+1}(\widetilde C)$ is an allowable cycle that represents a class in $\ih^{cl}_{m+1} (N )$. This gives the first formula. 

The second formula can be shown by a similar argument since $\ih^{c}_m (N )$ can be identified with $H_m (N') \simeq H_m (\p \widetilde M)$ divided by homologies arc-symmetric in a neighborhood of $x_0$.  These can be lifted to homologies in $\widetilde M$.  This shows that $\ih^{c}_m (N)$ is isomorphic to $\Image (\tilde i_m) \cong \Coker (\widetilde \p_{m+1} )$. This proves the lemma. 
\end{proof}

Now consider the following diagram with exact rows, where the vertical isomorphisms are given by the intersection pairing:
\begin{equation*}\begin{CD} 
H_{m+1}(\widetilde M) @>>>  H_{m+1} (\widetilde M,\p \widetilde M)  
@>\tilde \p_{m+1}>> H_{m} (\p  \widetilde M) @>>> 
H_{m}(\widetilde M) \\
@V\simeq VV  @V\simeq VV @V\simeq VDV @V\simeq VV \\
(H_{m} (\widetilde M,\p \widetilde M))^* @>>> 
(H_{m}(\widetilde M))^*  
@>>> (H_{m} (\p \widetilde M))^* 
@>(\tilde \p_{m+1})^*>> (H_{m+1}(\widetilde M, \p \widetilde M) )^*
\end{CD}\end{equation*}
By exactness, $D(\Image(\tilde \p_{m+1})) = \Ker((\tilde \p_{m+1})^*)$, which implies that $\Coker (\tilde \p_{m+1})$ and $\Image (\tilde\p_{m+1})$  are dually paired by the intersection product on $\p\widetilde M$. We have
\begin{equation*}\begin{CD}  
\Coker (\tilde \p_{m+1})  
@>D>\simeq> (\Image (\tilde\p_{m+1}))^*  \\ 
@VV\simeq V @VV\simeq V \\ 
\ih^c_m(N)  
@>D_m>> (\ih^{cl}_{m+1}(N))^*\\
 \end{CD}\end{equation*}
 and this diagram commutes by the argument given for \eqref{int1}. Therefore $D_m$ is an isomorphism. This completes the proof of Proposition \ref{localduality}.
\end{proof}

To prove global duality (Theorem \ref{duality}) using  local duality (Proposition \ref{localduality}) we use a Mayer-Vietoris argument. Let $Y$ be the set of singular points of $X$, and let $X'=X\setminus Y$. Let $N$ be the union of the open neighborhoods $N(x)$ for $x\in Y$, and let $N'= N\setminus Y$. We have $X'\cup N= X$ and $X'\cap N=N'$. This structure yields Mayer-Vietoris  sequences for $\ih^c_*(X)$ and $\ih^{cl}_*(X)$, with duality morphisms between them.
\begin{prop}
Let $X$ be a purely $n$-dimensional real algebraic variety with isolated singularities. There is a commutative diagram with exact rows ($k+l=n$):

\[\minCDarrowwidth15pt\begin{CD} 
@>>>\ih^c_k(N') @>>>  \ih^c_k (X')\oplus\ih^c_k(N)  
@>>> \ih^c_k(X)@>\Delta>> 
\ih^c_{k-1}(N')@>>> \\
@ . @VVD_{N'} V  @VVD_{X'}\oplus D_NV @VVD_XV @VVD_{N'}V@ . \\
@>>>(\ih^{cl}_l(N'))^* @>>> 
(\ih^{cl}_l(X'))^*\oplus (\ih^{cl}_l(N))^*
@>>> (\ih^{cl}_l(X))^* 
@>(\Delta')^*>> (\ih^{cl}_{l+1}(N'))^*@>>>
\end{CD}\]
\end{prop}
\begin{proof}
The construction of these exact sequences is analogous to the construction of the Mayer-Vietoris sequence for singular homology (\emph{cf}.\ \cite{massey}, ch.VII, \S 5). 

Let $IC^c_k(X)$ be the group of allowable $k$-chains of $X$ with respect to a good stratification $\mc S$. (We assume that $N$ is a union of strata of $\mc S$.) Let $IC^c_k(X') + IC^c_k(N)$ be the subgroup of $IC^c_k(X)$ consisting of all allowable chains that are sums of allowable chains of $X'$ and allowable chains of $N$. We claim that $IC^c_k(X') + IC^c_k(N)=IC^c_k(X)$; from this fact the classical construction gives the Mayer-Vietoris sequence.

To prove the claim, let $[C]\in IC^c_k(X)$, and choose $\zeta>0$ such that for all singular points $x_0\in Y$, we have $(S(x_0,\zeta)\cap X)\subset N$, and $S(x_0,\zeta)$ is transverse to $\mc S$ and $C$. Let $C_N = \bigcup_{x_0\in Y} (C\cap B(x_0,\zeta))$ and $C_{X'}= cl(C\setminus C_N)$. Then $[C_{X'}]\in IC^c_k(X')$, $[C_N]\in IC^c_k(N)$, and $[C]= [C_{X'}]+[C_N]$.

The boundary map $\Delta$ can be described as follows. The classical construction is this: If $C$ is a cycle in $X$, write $C=C_1+C_2$, where $C_1$ is a chain in $X'$ and $C_2$ is a chain in $N$. Then $\Delta C = \p C_1 =\p C_2$. So $\Delta$ has the following geometric description. Choose $\eta>0$ such that for all singular points $x_0\in Y$, we have $(S(x_0,\eta)\cap X)\subset N$, and $S(x_0,\eta)$ is transverse to $\mc S$. Let $L'(x_0) = S(x_0,\eta)\cap X$ and let $L' = \bigcup_{x_0} L'(x_0)$.
Then $\Delta[C] = [C'\cap L']$ where $[C']=[C]$ and $C'$ is transverse to $L'$.

The construction of the Mayer-Vietoris seqence for intersection homology with closed supports is parallel. There are restriction maps $IC^{cl}_l(X')\to IC^{cl}_l(N')$ and $IC^{cl}_l(N)\to IC^{cl}_l(N')$, and $IC^{cl}_l(N')= IC^{cl}_l(X')+IC^{cl}_l(N')$. The geometric description of $\Delta'$ is the same as for $\Delta$.

The commutativity of the ladder of duality maps follows from the geometric descriptions of $\Delta$  and the definition of the intersection pairing.
\end{proof}

To complete the proof of Theorem \ref{duality}, note that the maps $D_{X'}$ and $D_{N'}$ are isomorphisms since $X'$ and $N'$ are nonsingular, and the maps $D_N$ are isomorphisms by Proposition \ref{localduality}. Therefore the maps $D_X$ are isomorphisms.
\end{proof}

An alternate proof of Theorem \ref{duality} is to compute $\ih^c_*(X)$ and $\ih^{cl}_*(X)$ directly from the definitions, bypassing the local computations. The results of these global computations can be summarized as follows. (The proof is very similar to the proof of Proposition \ref{localduality}.)

Let $X$ be a purely $n$-dimensional real algebraic variety with isolated singularities, and let $Y$ be the set of singular points of $X$. Let $N$ be a small semialgebraic open neighborhood of $Y$ as above. Let $\pi:\widetilde X\to X$ be a resolution of singularities, and let $\widetilde N = \pi^{-1}(N)$.

If $n=2m$ we have
\begin{equation}\label{even}
\ih^c_k(X) = 
\begin{cases}
H_k(X),\ \ k>m\\
\Image[H_k(X\setminus N)\to  H_k(X)],\ \ k=m\\
H_k(X\setminus N),\ \ k<m
\end{cases}
\end{equation}
If $n=2m+1$ we have
\begin{equation}\label{odd}
\ih^c_k(X) = 
\begin{cases}
H_k(X),\ \ k>m+1\\
\Image[H_k(\widetilde X)\to  H_k(X)],\ \ k=m+1\\
\Image[H_k(\widetilde X\setminus \widetilde N)\to  H_k(\widetilde X)],\ \ k=m\\
H_k(X\setminus N),\ \ k<m
\end{cases}
\end{equation} 
\noindent
The same computations hold for $\ih^{cl}_*(X)$. (Note that the inclusions $X\setminus N \to X$ and $\widetilde X\setminus \widetilde N \to \widetilde X$ are proper.)

Let $M= X\setminus N$, a semialgebraic manifold with boundary. We have semialgebraic homeomorphisms $\cl(N)\cong c(\p M)$ (the cone on $\p M$), $X\cong M\cup_{\p M} c(\p M)$, and $\widetilde X\cong M \cup _{\p M} \cl(\widetilde N)$.
Duality for the intersection homology of $X$ follows from   formulas \eqref{even} and \eqref{odd} for compact supports, the corresponding formulas for closed supports, and classical duality for manifolds with boundary.

\section{Appendix: the pseudoboundary}

Let $C$ be a $k$-dimensional semialgebraic closed subset of the real algebraic variety $X$. We define the \emph{pseudoboundary} of $C$ to be the closed semialgebraic set
$$
\Sigma C = \{x\in C\ ; \text { $C$ is not arc-symmetric in 
dimension $k$ at $x$}\}.
$$

\begin{prop}\label{boundary}
$\partial C \subset \Sigma C$
\end{prop}
\begin{proof}
Let $Z$ be the Zariski closure of $C$ in $X$ and let $\pi : \widetilde Z\to Z$ be a resolution of singularities of $Z$ adapted to $C$ 
(see \cite {mccpar2}, p.\ 134).   We assume that there is an algebraic subset $Y\subset Z$, with $\dim Y< k=\dim Z$, such that $E = \pi^{-1} (Y)$ is a normal crossing divisor, and $\pi$ induces an isomorphism between $\tilde Z \setminus E$ and $Z\setminus Y$.  
 We let $\widetilde C=\cl(\pi^{-1}(C)  \setminus E)$ be the  pullback of $C$ by $\pi$ (\cite {mccpar2}, p.\ 157).  In local coordinates adapted 
to the divisor with normal crossings $E$, we have that $\widetilde C$ is a union of closed quadrants.  We have 
$ \p \widetilde C= \widetilde C\cap \cl(\widetilde Z \setminus \widetilde C)$, the topological boundary of $\widetilde C$.

\begin{lem}\label{resolution}
$\p \widetilde C = \Sigma \widetilde C$ and 
$\Sigma C = \pi (\Sigma \widetilde C)$.
\end{lem}

\begin{proof}The first claim is obvious and the second one follows  easily from the fact that $\pi$ and the strict transform by $\pi$ give a one-to-one correspondence between the analytic arcs 
on $Z$ not entirely included in $Y$ and the analytic arcs in $\widetilde Z$ not entirely included in $E$.  

Indeed,  if $x\in \Sigma C$  there is an analytic arc $\gamma :(-\varepsilon, \varepsilon) \to Z$, not entirely included in $\pi(E)$, such that 
$\gamma (-\varepsilon , 0) \subset C$ and $\gamma (0, \varepsilon ) \subset Z \setminus C$.  Since $\gamma$ can be lifted to $\tilde \gamma :(-\varepsilon, \varepsilon) \to \widetilde Z$ with $\tilde \gamma (-\varepsilon , 0) \subset \widetilde C$ and 
$\tilde \gamma (0, \varepsilon ) \subset \widetilde Z \setminus \widetilde C$, we have $\tilde \gamma (0) \in \Sigma \widetilde C$.  This shows that 
$\Sigma C \subset \pi (\Sigma  \widetilde C)$.  

Conversely, if $x=\pi(z)$, $z\in \Sigma \widetilde C$, then there is an analytic arc $\tilde \gamma :(-\varepsilon, \varepsilon) \to \widetilde Z$ such that $\tilde \gamma (-\varepsilon , 0) \subset \widetilde C$ and $\tilde \gamma (0, \varepsilon ) \subset \widetilde Z \setminus \widetilde C$, with $\tilde \gamma (0)=z$.  Moreover, for every  semialgebraic $V\subset C$, $\dim V <k$, the arc
$\tilde \gamma$ can be chosen so that the image of $\tilde \gamma$ intersects $\pi^{-1}(V)$ only at $z$.  Then 
$\gamma =  \pi \circ \tilde \gamma$ satisfies $\gamma (-\varepsilon , 0) \subset C$ and $\gamma (0, \varepsilon ) \subset Z \setminus C$.  This shows that $x\in \Sigma C$, and hence we conclude that $\pi ( \Sigma  \widetilde C)\subset  \Sigma C$.
\end{proof}

To prove the Proposition, note that 
$[C] = \pi_*[ \widetilde   C] $  and $[\p C] = \pi_*[\p  \widetilde C] $.  Therefore   $\supp C \subset \pi (\supp \widetilde C)$ and 
$\p C \subset \pi (\p\widetilde C) = \pi (\Sigma  \widetilde C) = \Sigma C .
$
\end{proof}

\begin{rem}
The adapted resolution construction induces a filtration on the complex of semialgebraic chains; see  \cite {mccpar2}. 
\end{rem}

\begin{question}
Is it always possible after performing more blow-ups to obtain $(\widetilde C, \partial\widetilde C)=(V,\p V)$, a manifold with corners?
\end{question}

\begin{prop}\label{image}
Let  $f:X\to Y$ be a regular morphism of real algebraic varieties, and let $C$ be a $k$-dimensional semialgebraic closed subset of $X$. Suppose that $f|C$ is proper.  Then $\Sigma  f_*(C ) \subset f  (\Sigma C)$.    
\end{prop}

\begin{proof}
We apply Lemma \ref{resolution} to both $C\subset X$ and $D ={ f_*(C  )}  $. (We assume $f_*([C]) \ne  0$ otherwise the claim is obvious.)  Let  $\pi _C :\widetilde Z \to Z\subset X$ and $\pi _D :\widetilde  W\to W \subset Y$ be  adapted 
resolutions of $C$ and $D$, respectively.  

We assume for simplicity that the image of every irreducible component of $Z$ by $f$ is of dimension $k$.  
Then, after performing more blow-ups if necessary, there is a regular morphism $\tilde f$ so that the following diagram 
commutes:
  \[
  \xymatrix{
     \tilde Z \ar^{\widetilde f}[rr] \ar_{\pi_C }[d] && \widetilde W \ar^{\pi_D} [d] \\
    Z  \ar^{f }[rr]  && W 
  } 
\]  
Note that   $\widetilde D=  \tilde f_* (\widetilde C) $  gives  $\p \widetilde D=  \tilde f_* (\p \widetilde C) $, which  implies $\p \widetilde D 
\subset   \tilde f (\p \widetilde C) $.  
Now 
\begin{align*}
\Sigma f_* (C) & = \Sigma D = \pi_D (\p \widetilde D)  \\ \notag
 & \subset   \pi_D (\tilde f(\p \widetilde C) ) = 
\pi_D(\tilde f(\Sigma \widetilde C)) =  f (\pi_{C} ( \Sigma \widetilde C) )= f (\Sigma C) .
\end{align*}
\end{proof}

For several applications we need the following property of the pseudoboundary of a deformation homology \eqref{defhom} with respect to a deformation $\Psi_t$ of the identity.

\begin{prop}\label{defchain}
Let $X$ be a real algebraic variety, and let $U$ be a neighborhood of the origin in $\R^l$.
Let $\Phi(t,x) = (t, \Psi (t,x)) : U\times X \to U\times X$  be an arc-analytic and semialgebraic homeomorphism with arc-analytic and semialgebraic inverse, such that $\Psi(0,x) = x$ for all $x\in X$. Let $t_0\in U$ be such that the segment $I=\{t\in\R^l\ ;\ t = st_0, 0\leq s\leq 1\}$ is contained in U. Let $C$ be a $k$-cycle in $X$, and let $B = \Psi_* (I\times C)$. Then 
$$\Sigma B \subset \Psi ( \Sigma (I\times C) )= C\cup \Psi  (I\times \Sigma C)\cup \Psi  (\{t_0\} \times C).$$ 
\end{prop}

\begin{proof}
We apply Proposition \ref{image}  to the chain $\Phi (I\times C) $ and the projection on the second factor 
$f:I\times X\to X$.  Then $B = f_* (\Phi (I\times C) ) $ and by  the properties of $\Phi$ we have $\Sigma \Phi (I\times C) = \Phi (\Sigma ( I\times C)) $.  So by Proposition \ref{image}, 
$
\Sigma B \subset  f (\Sigma \Phi (I\times C) )  = f( \Phi (\Sigma ( I\times C))) =  \Psi  (I\times \Sigma C). 
$
\end{proof}


\end{document}